\documentclass{article}
\usepackage[utf8]{inputenc}
\pdfoutput=1
\usepackage[english]{babel}
\usepackage[bottom =2.5cm, top= 2.5cm, left= 2.5cm, right= 2.5cm]{geometry}
\usepackage{amsmath, amsfonts, amssymb}
\usepackage{amsthm}
\usepackage{graphicx}
\usepackage{indentfirst}
\newtheorem{theorem}{Theorem}[section]
\newtheorem{defi}[theorem]{Definition}
\newtheorem{prop}[theorem]{Proposition}
\newtheorem{exem}[theorem]{Example}
\newtheorem{obs}[theorem]{Remark}
\newtheorem{cor}[theorem]{Corollary}

\usepackage{float}
\usepackage{tikz}
\usetikzlibrary{shapes,shapes.geometric,arrows,fit,calc,positioning,automata}
\date{\today}

\begin{document}

\title{Representations of C*-algebras of row-countable graphs and unitary equivalence}

\author{Ben-Hur Eidt\footnote{This author is partially supported by Conselho Nacional de Desenvolvimento Cient\'{i}fico e Tecnol\'{o}gico - CNPq. } and Danilo Royer}
\maketitle

\begin{abstract}
    In this article we generalize the main results of \cite{A} and \cite{B}. More specifically, we show that there are branching systems (which induce representations of the graph $C^*(E)$) associated to each row-countable graph $E$. For row-countable graphs, we characterize the condition $(L)$ via branching systems. Moreover, we show that each permutative representation in Hilbert spaces operators is unitarily equivalent to one induced by a branching system, even the spaces being not separable. Furthermore, under some hypothesis on the graph, we show that each representation of the graph C*-algebra is permutative.
    \end{abstract}
\vspace{1 cm}

\section*{Introduction}

The concept of a graph C*-algebra was first developed, in \cite{origemce}, by considering row-finite countable graphs (recall that a graph $E=(E^0,E^1,r,s)$ is countable if $E^0$ and $E^1$ are both countable and is row-finite if $s^{-1}(v)$ is finite for each vertex $v$), and have been extensively explored since then. 

Ideas related to branching systems have been studied in some areas like random walks, symbolic dynamics, scientific computing and operator theory, see \cite{B} for references.

In this paper we deal with branching systems in row-countable graphs, that is, graphs with the property that $s^{-1}(v)$ is at most countable for each vertex $v$. In \cite{A} the authors define a structure called branching system for graphs and show how to obtain a representation of $C^*(E)$ through a branching system of a graph $E$. Moreover, there is proved a result that ensures the existence of a branching system for all countable graphs. We prove this theorem for a larger class of graphs, the row-countable graphs. In \cite{B}, it is proved that each permutative representation $\varphi: C^*(E) \to B(H)$ (with $H$ separable) is unitarily equivalent to a representation arising from a branching system. We prove this result even $H$ being not separable. Moreover, in \cite{B}, the authors find a class of graphs where each representation $\varphi: C^*(E) \to B(H)$ (with $H$ separable) is permutative. We find a larger class where this result remains to be true.

The paper is organized as follows. In the first chapter we introduce branching systems and recall from \cite{A} how to obtain representations through this structure. After this we show how to obtain branching systems for row-countable graphs, and for graphs of this class, we characterize the condition $(L)$ via branching systems. In the second chapter, we consider a permutative representations $\varphi:  C^*(E) \to B(H)$ and show that this representations are unitarily equivalent to one induced by a branching system, even $H$ not being separable. In the last chapter, we prove that for graphs in a certain class, each representation is permutative. 

In this work, following \cite{ianr}, given an arbitrary graph $E$ we define the algebra $C^*(E)$ as being the universal $C^*$-algebra generated by  $\{P_v\}_{v \in E^0} \cup \{S_e\}_{e \in E^1}$ with the following  relations: $\{P_v\}_{v \in E^0}$ are mutually orthogonal projections,  $\{S_e\}_{e \in E^1}$ are partial isometries with orthogonal ranges and

\begin{itemize}
\item[(CK1)] $S_e^*S_e = P_{r(e)}$ $\forall e \in E^1$,
\item[(CK2)] $P_{s(e)}S_eS_e^* = S_eS_e^*$ $\forall e \in E^1$,
\item[(CK3)] $P_v = \sum\limits_{e \in s^{-1}(v)}S_eS_e^*$ provided that $v \in E^0$ is such that $0 < \# s^{-1}(v) < \infty$.
\end{itemize}

\section{Representations arising from branching systems}

In this section we define $E$-branching systems and recall from \cite{A} a theorem that shows how obtain a representation induced from a branching system. After that we prove that for row-countable graphs graphs there always exists a branching system. For graphs in this class, we also characterize the condition $(L)$ via representations induced from branching systems.

\begin{defi}\label{ESR}[2.1:\cite{A}]
Let $E = (E^0, E^1, r,s)$ be a graph, $(X, \mathcal{M}, \mu)$ a measure space and $\{R_e\}_{e \in E^1}$, $\{D_v\}_{v \in E^0}$ a collection of measurable subsets of $X$ such that:

\begin{enumerate}
\item $R_e \cap R_f \overset{\mu -a.e.}{=} \varnothing$, $\forall e,f \in E^1$ with $e \neq f$.
\item $D_v \cap D_w \overset{\mu -a.e.}{=} \varnothing$, $\forall v,w \in E^0$ with $v \neq w$.
\item $R_e \overset{\mu -a.e.}{\subseteq} D_{s(e)}$, $\forall e \in E^1$.
\item $D_v \overset{\mu -a.e.}{=} \bigcup \limits_{e \in s^{-1}(v)} R_e$, $\forall v \in E^0$ such that $0 < \# s^{-1}(v) < \infty$.
\item For each $e \in E^1$ there exist functions $f_e: D_{r(e)} \to R_{e}$, $f_e^{-1}: R_e \to D_{r(e)}$, both measurable and such that $f_e(D_{r(e)}) \overset{\mu -a.e.}{=} R_e$, $f_e^{-1} \circ f_e \overset{\mu -a.e.}{=} Id_{D_{r(e)}}$ e $f_e \circ f_e^{-1} \overset{\mu -a.e.}{=} Id_{R_e}$.
\item For each $e \in E^1$ there exist the Radon-Nikodym derivatives $\dfrac{d \mu \circ f_e}{d \mu}$ and $\dfrac{d \mu \circ f_e^{-1}}{d \mu}$, denoted by $\Phi_{f_e}$ and $\Phi_{f_e^{-1}}$, respectively . Furthermore, $$( \Phi_{f_e^{-1}} \circ f_e ) . \Phi_{f_e} = 1 = ( \Phi_{f_e} \circ f_e^{-1} ) . \Phi_{f_e^{-1}} \,\ \mu -a.e.$$
\end{enumerate}

The measure space $(X, \mathcal{M}, \mu)$, with the collections $\{R_e\}_{e \in E^1}$, $\{D_v\}_{v \in E^0}$ and the functions
 $f_e$, $f_e^{-1}$, $\Phi_{f_e}$,  $\Phi_{f_e^{-1}}$, satisfying the items above is called an $E$-\textbf{branching system}.
\end{defi}

In the measure spaces $D_{r(e)}$ and $R_e$ we consider the $\sigma$-algebras induced by $X$, moreover, since the Radon-Nikodym derivative is a positive function, it follows from item $6$ that $\Phi_{f_e}, \Phi_{f_e^{-1}} > 0$ $\mu- a.e$. Below we show a sufficient condition that ensures the equality of item 6.

\begin{prop}\label{mq0} Let $(X,\mathcal{M},\mu)$ a measure space and suppose that the items $1$ until $5$ from definition \ref{ESR} are satisfied and exist $\Phi_{f_e}$ and $\Phi_{f_e}^{-1}$. If for each $e \in E^1$ the measures $\mu: D_{r(e)} \to [0,\infty]$ and $\mu: R_e \to [0,\infty]$ are semi-finite then $(\Phi_{f_e^{-1}} \circ f_e) . \Phi_{f_e} = 1 = ( \Phi_{f_e} \circ f_e^{-1} ) . \Phi_{f_e^{-1}}$ \,\ $\mu- a.e$. In particular, $(X,\mathcal{M},\mu)$ is a $E$-branching system.
\end{prop}

\begin{proof}
For each measurable set $E$,
$$\int\limits_{E} ( \Phi_{f_e^{-1}} \circ f_e ) . \Phi_{f_e} \,\ d\mu = \int\limits_{D_{r(e)}} \chi_{E} . ( \Phi_{f_e^{-1}} \circ f_e ) . \Phi_{f_e} \,\ d\mu =  \int\limits_{D_{r(e)}} \chi_{E} . ( \Phi_{f_e^{-1}} \circ f_e ) \,\ d(\mu \circ f_e) =$$ $$= \int\limits_{R_e} ( \chi_{E} \circ f_e^{-1}) . \Phi_{f_e^{-1}} \,\ d\mu = \int\limits_{R_e} ( \chi_{E} \circ f_e^{-1} ) \,\ d(\mu \circ f_e^{-1}) = \int\limits_{D_{r(e)}} \chi_{E} \,\ d\mu = \int\limits_{E} 1 \,\ d\mu.$$

Since $\mu$ is semi-finite then $(\Phi_{f_e^{-1}} \circ f_e) . \Phi_{f_e} \overset{\mu -a.e.}{=} 1$. The same argument shows that $( \Phi_{f_e} \circ f_e^{-1} ) . \Phi_{f_e^{-1}} = 1$ \,\ $\mu- a.e$. 
\end{proof}

In the article \cite{A}, two questions were developed, the first is the connection between a branching system and representations of $C^*(E)$. The main result in this way is the following theorem.

\begin{theorem}\label{rep}[2.2: \cite{A}]
Let $E = (E^0, E^1, r, s)$ be a graph and $(X, \mathcal{M}, \mu)$ a $E$-branching system. Then exists a representation $\pi: C^*(E) \to B(L^2(X, \mathcal{M}, \mu))$ such that
$$\pi(S_e)(\phi) = \chi_{R_e} .\Phi_{f_e^{-1}}^{\frac{1}{2}} . (\phi \circ f_e^{-1}) \,\ \,\ \,\  \textrm{e} \,\ \,\ \,\ \pi(P_v)(\phi) = \chi_{D_v} . \phi$$
\end{theorem}

The second question is: given a graph $E$, there exists always an $E$-branching system? There is shown that if $E$ is countable then the answer is positive, that is, there always exists an $E$-branching system (see \cite[Theorem 3.1]{A}). However, the hypothesis can be weakened and that is the main goal of this section. For this, we need some preliminary results.

Let $E = (E^0, E^1, r, s)$ be a graph. Let $(0,1]$ with the Borel $\sigma$-algebra and the Lebesgue measure and let
and $\Lambda := E^0 \cup E^1$. Define for each $A \subseteq (0,1] \times \Lambda$ and for each $\lambda \in \Lambda$ the set $A^{(\lambda)} := \{t \in (0,1] \,\ | \,\ (t, \lambda) \in A\}$ and define the collection $\mathcal{M} = \{A \subseteq (0,1] \times \Lambda \,\ | \,\ A^{(\lambda)} \textrm{ is Borel measurable } \forall \lambda \in \Lambda\}$.
Moreover, define the function $\mu: \mathcal{M} \to [0,\infty]$ given by $\mu(A) = \sum\limits_{\lambda \in \Lambda} m(A^{(\lambda)})$ where $m$ is the Lebesgue measure.

\begin{prop}
$( (0,1] \times \Lambda, \mathcal{M}, \mu )$ is a measure space.
\end{prop}
\begin{proof}
The reader can check that $\mathcal{M}$ is a $\sigma$-algebra. Let $\{A_n\}_{n \in \mathbb{N}} \subseteq \mathcal{M}$ be a family of disjoints subsets and $A = \bigcup\limits_{n \in \mathbb{N}} A_n$. We will consider two cases.

\textit{Case 1}: Suppose that $m(A^{(\lambda)}) > 0$ for a uncountable number of elements $\lambda \in \Lambda$. Then  $$\mu\left(\bigcup\limits_{n \in \mathbb{N}} A_n\right) = \mu(A) = \sum\limits_{\lambda \in \Lambda} m(A^{(\lambda)}) = \infty.$$
On the other hand, there exist $n_0 \in \mathbb{N}$ with the property that there exists an uncountable number of elements, $\lambda \in \Lambda$, such that $m(A_{n_0}^{(\lambda)}) > 0$. In fact, otherwise,  for each $\mathbb{N}$, there exist at most a countable number of elements $\lambda \in \Lambda$ such that $m(A_{n}^{(\lambda)}) > 0$ and then there exists a countable number of elements $\lambda \in \Lambda$ such that $m \left(\bigcup\limits_{n \in \mathbb{N}} A_n^{(\lambda)}\right) > 0$.  Since $A^{(\lambda)} = \bigcup\limits_{n \in \mathbb{N}} A_n^{(\lambda)}$ we conclude that $m(A^{(\lambda)}) > 0$ for a countable number of elements $\lambda \in \Lambda$, and this is a contradiction with the hypothesis assumed in the \textit{Case 1}. Therefore,
 $$\mu(A_{n_0}) = \sum\limits_{\lambda \in \Lambda} m(A_{n_0}^{(\lambda)}) = \infty$$ and as $\sum\limits_{n \in \mathbb{N}} \mu(A_n) \geq \mu(A_{n_0}) = \infty$ then $$\sum\limits_{n \in \mathbb{N}} \mu(A_n)  = \mu\left(\bigcup\limits_{n \in \mathbb{N}} A_n\right).$$

\textit{Case 2}: Suppose that $m(A^{(\lambda)}) > 0$ for, at most, a countable number of elements $\lambda \in \Lambda$.
In this case,
$$\mu\left(\bigcup\limits_{n \in \mathbb{N}} A_n\right) = \sum\limits_{\lambda \in \Lambda} m\left( \left(\bigcup\limits_{n \in \mathbb{N}} A_n\right)^{(\lambda)} \right) = \sum\limits_{\lambda \in \Lambda} m\left(\bigcup\limits_{n \in \mathbb{N}} A_n^{(\lambda)}\right)= $$ $$= \sum\limits_{\lambda \in \Lambda} \sum\limits_{n \in \mathbb{N}} m(A_n^{(\lambda)}) = \sum\limits_{n \in \mathbb{N}}\sum\limits_{\lambda \in \Lambda}m(A_n^{(\lambda)}) = \sum\limits_{n \in \mathbb{N}} \mu(A_n).$$
    This finishes the proof.\end{proof}

\begin{defi}
 Let $E = (E^0, E^1, r,s)$ be a graph. We say that $E$ is row-countable if $s^{-1}(v)$ is countable for every $v \in E^0$.
\end{defi}

We are now ready to prove the main theorem of this section.
\begin{theorem}\label{princ}
Let $E = (E^0, E^1, r,s)$ be a row-countable graph. Then there exists an $E$-branching system.
\end{theorem}
\begin{proof}
Let $((0,1] \times \Lambda ,\mathcal{M}, \mu)$ be the measure space as in the previous proposition. Note that if $I \subseteq (0,1]$ is Borel-measurable and $(\lambda_n)_{n \in \mathbb{N}} \subseteq \Lambda$ then $I \times \{\lambda_1, \lambda_2, \ldots \} \in \mathcal{M}.$
For each $e \in E^1$ define $R_e = (0,1] \times \{e\}$, let $W = \{v \in E^0\,\ | \,\ v \textrm{ is a sink }\}$. For each $v \in W$ define $D_v = (0,1] \times \{v\}$ and for each $v \in E^0 - W$ define $$D_v := \bigcup\limits_{e \in s^{-1}(v)} R_e = \bigcup\limits_{e \in s^{-1}(v)} (0,1] \times \{e\} = (0,1] \times  s^{-1}(v).$$
Now we verify the first four conditions of Definition \ref{ESR}, and after that we define the maps $f_e:D_{r(e)}\rightarrow R_e$.

For $e,d \in E^1$ with $d \neq e$ it holds that $$R_d \cap R_e = (0,1] \times \{d\} \cap (0,1] \times \{e\} = (0,1] \times \{e\}\cap\{d\} = \emptyset.$$ 

Fore $v,u \in E^0$ with $v \neq u$ there are three cases:
\begin{itemize}
\item If $u,v \in W$ then $$D_v \cap D_u = (0,1] \times \{v\} \cap (0,1] \times \{u\} = (0,1] \times \{v\}\cap\{u\} = \emptyset.$$
\item If $u \in W$ and $v \notin W$ then $$D_v \cap D_u = (0,1] \times \{v\} \cap (0,1] \times s^{-1}(u) = (0,1] \times \{v\} \cap s^{-1}(u) = \emptyset.$$
\item If $u \notin W$ and $v \notin W$ then $$D_v \cap D_u = (0,1] \times s^{-1}(v) \cap (0,1] \times s^{-1}(u) = (0,1] \times s^{-1}(v) \cap s^{-1}(u) = \emptyset.$$
\end{itemize}
Moreover, given $e \in E^1$, since $s(e)$ is not a sink we have $R_e \subseteq \bigcup\limits_{d \in s^{-1}(s(e))} R_d = D_{s(e)}$. If $v \in E^0$ is such that $0 < \#s^{-1}(v) < \infty$, $v$ is not a sink and then $D_v = \bigcup\limits_{e \in s^{-1}(v)}R_e$.

Now we define the functions $f_d: D_{r(d)} \to R_d$ for each $d \in E^1$. 
Fix $d \in E^1$.
\begin{itemize}
\item \textit{Case 1: $r(d) \in W$}. In this case $D_{r(d)} = (0,1] \times \{r(d)\}$ e $R_d = (0,1] \times \{d\}$. Define
$$f_d: (0,1] \times \{r(d)\} \to (0,1] \times \{d\}$$
$$(t,r(d)) \mapsto (t,d).$$
Clearly $f_d$ is a bijection and, if $f_d^{-1}$ is the inverse of $f_d$ then $f_d \circ f_d^{-1} = Id_{R_d}$, $f_d^{-1} \circ f_d = Id_{D_{r(d)}}$ and $f(D_{r(d)}) = R_d$. Note that $f_d$ are measurable: let $B \in \mathcal{M}_{R_d}$ then $B = A \cap R_d = A \cap (0,1] \times \{d\}$ for some $A \in \mathcal{M}$ and then $B = A^{(d)} \times \{d\}$. Therefore
$$f_d^{-1}(B) = f_d^{-1} ( A^{(d)} \times \{d\} ) = A^{(d)} \times \{r(d)\} \in \mathcal{M}.$$
As $f_d^{-1}(B) \subseteq D_{r(d)}$ then $f_d^{-1}(B) \in \mathcal{M}_{D_{r(d)}}$ and then $f_d$ is measurable. The same argument shows that $f_d^{-1}$ is measurable.

Moreover, as $\mu(D_{r(d)}) = \mu( (0,1] \times \{r(d)\} ) = m(0,1] = 1 < \infty$ and $\mu \circ f_d (D_{r(d)}) = \mu(R_d) = \mu( (0,1] \times \{d\}) = m(0,1] = 1 < \infty$ then the measures are $\sigma$-finite. If $B \in M_{D_{r(d)}}$ is such that $\mu (B) = 0$ then $B = A \cap (0,1] \times \{r(d)\} = A^{(r(d))} \times \{r(d)\}$ for some $A \in \mathcal{M}$ and then $0 = \mu(B) = m(A^{(r(d)})$. Therefore
$$\mu \circ f_d (B) = \mu \circ f_d ( A^{(r(d))} \times \{r(d)\}) = \mu (A^{(r(d))} \times \{d\}) = m(A^{(r(d))}) = 0$$
and then $\mu \circ f_d \ll \mu$. By the Radon-Nikodym Theorem there exists $\dfrac{d \mu \circ f_d}{d \mu} =: \Phi_{f_d}$.  The same argument can be used for $\Phi_{f_d^{-1}}$. By Proposition \ref{mq0} we conclude that the sixth condition of Definition \ref{ESR} is true.

\item \textit{Case 2: $0 < \#s^{-1}(r(d)) < \infty$}.

Write $s^{-1}(r(d)) = \{e_1, \ldots, e_N\}$ for some $N \in \mathbb{N}$. Note that $(0,1] = \bigcup\limits_
{j = 1}^{N} I_j$ where, for each $j = 1,\ldots, N$  $I_j = (\frac{j-1}{N}, \frac{j}{N}]$. 
Fix $j \in \{1,\ldots, N\}$. Let $\varphi_j: (0,1] \to I_j$ be a homeomorphism such that $\varphi_{j |(0,1)} : (0,1) \to (\frac{j-1}{N}, \frac{j}{N})$ is a diffeomorphism (for example, the linear increasing homeomorphism) and define 
$$f_{d,j}: (0,1] \times \{e_j\} \to I_j \times \{d\}$$ 
$$(t,e_j) \mapsto (\varphi_j(t),d)$$

Notice that $f_{d,j}$ is a bijection with inverse given by $f_{d,j}^{-1}(t,d) = (\varphi_j^{-1}(t), e_j)$ for each $t \in I_j$. Moreover, $f_{d,j}$ is measurable (this follows from a similar argument that was used in Case 1 and from the fact that $\varphi_j$ is continuous). Define

$$f_d: (0,1] \times \{e_1,\ldots,e_N\} \to  (0,1] \times \{d\}$$
$$(t,e_j) \mapsto f_{d,j}(t,e_j)$$

for each $j = 1, \ldots, N$. Note that $f_d$ is bijetction; as $(0,1] \times \{e_1, \ldots, e_N\} = (0,1] \times \{e_1\} \cup \ldots \cup (0,1] \times \{e_n\}$ and $(f_d )_{|(0,1] \times \{e_j\}} = f_{d,j}$ for each $j = 1, \ldots, N$  the function $f_d$ is measurable. The same argument ensure that $f_d^{-1}$ is measurable. 

As $D_{r(d)} = \bigcup\limits_{i = 1}^{N} \left( (0,1] \times\{e_i\} \right)$, then $\mu \left( (0,1] \times \{e_i\} \right) = 1$ for each $i=1,\ldots,N$ and  $\mu \circ f_d (D_{r(d)}) = \mu( (0,1] \times \{d\})= 1$. Therefore both the measures are $\sigma$-finite.

If $B \in \mathcal{M}_{D_{r(d)}}$ then $B = A \cap {D_{r(d)}} = A^{(e_1)} \times \{e_1\} \cup \ldots \cup A^{(e_N)} \times \{e_N\}$ for some $A \in \mathcal{M}$. If $B$ satisfies $\mu(B) = 0$ then $\sum\limits_{i = 1}^{N} m(A^{(e_i)}) = 0$ and so $m(A^{(e_i)}) = 0$ for each  $i=1,\ldots, N$. Moreover,
$$f_d(B) = \bigcup\limits_{i = 1}^N f_d( A^{(e_i)} \times \{e_i\} )  = \bigcup\limits_{i = 1}^N f_{d,i}( A^{(e_i)} \times \{e_i\}).$$
For each $i = 1,\ldots, N$, 
$$f_{d,i}( A^{(e_i)} \times \{e_i\} ) = \{ f_{d,i}(t,e_i) \,\ | \,\ t \in  A^{(e_i)} \} = \{ (\varphi_i(t), d) \,\ | \,\ t \in  A^{(e_i)} \}$$
and then $$\mu ( f_{d,i}( A^{(e_i)} \times \{e_i\} ) ) = m( \varphi_i(A^{(e_i)} ) ) = m( \varphi_i( A^{(e_i)} \cap (0,1) \cup A^{(e_i)} \cap \{1\} )) =$$
$$= m( \varphi_i( A^{(e_i)} \cap (0,1) ) + m( \varphi_i( A^{(e_i)} \cap \{1\} ) )$$
Note that $A^{(e_i)} \cap (0,1) \subseteq A^{(e_i)}$ is $\mu$-null. As $\varphi_i$ is a diffeomorphism, it follows from \cite[pg. 153]{rudin} that $m( \varphi_i( A^{(e_i)} \cap (0,1) ) = 0$. Clearly $m( \varphi_i( A^{(e_i)} \cap \{1\} ) ) = 0$. Therefore $\mu (f_{d,i}( A^{(e_i)} \times \{e_i\} ) )= 0$ for each $i=1,\ldots,N$ and this shows that $\mu (f_d(B)) = 0$. Therefore $\mu \circ f_d \ll \mu$. By the Radon-Nikodym Theorem and by Proposition \ref{mq0} we conclude this case.  

\item {Case 3: $\#s^{-1}(r(d)) = \infty$}.

By hypothesis $s^{-1}(r(d))$ is countable and then we can write $s^{-1}(r(d)) = \{e_1,e_2, \ldots \}$. Moreover, $(0,1] = \bigcup\limits_{i = 1}^{\infty} I_j$ where $I_j = (\frac{1}{j - 1}, \frac{1}{j}]$. Fix $j \in \mathbb{N}$ and let $\varphi_j: (0,1] \to I_j$ a homeomorphism such that $\varphi_{j |(0,1)}: (0,1) \to (\frac{1}{j - 1}, \frac{1}{j})$ is a diffeomorphism (for example, the linear homeomorphism). Define
$$f_{d,j}: (0,1] \times \{e_j\} \to I_j \times \{d\}$$ 
$$(t,e_j) \mapsto (\varphi_j(t),d)$$
and
$$f_d : D_{r(d)} = (0,1] \times \{e_1, e_2, e_3, \ldots \} \to (0,1] \times \{d\}$$ 
$$(t,e_j) \mapsto (\varphi_j(t),d).$$

The same arguments used in Case 2 can be used in Case 3 and so we get an $E$-branching system.
\end{itemize}
\end{proof}

Note that if a $E$ graph is countable or row-finite then, in particular, $E$ is row-countable. Therefore for this kind of graphs we can ensure the existence of a branching system. 

\begin{cor}
Let $E = (E^0, E^1, r,s)$ be a row-countable graph. Then, for every $v \in E^0$ and $e \in E^1$ it holds that $P_v \neq 0$ and $S_e \neq 0$.
\end{cor}
\begin{proof}
As $s^{-1}(v)$ is countable for every $v \in E^0$, let $X$ a $E$-branching system. Note that the representation $\pi: C^*(E) \to B(L^2(X))$ induced by Theorems \ref{princ} and \ref{rep} is such that $\pi(P_v)(\phi) = \chi_{D_v} . \phi$ and by the proof of Theorem \ref{princ} we get that $\mu(D_v) > 0$. Therefore $\pi(P_v) \neq 0$ and then $P_v \neq 0$. From $S_e^*S_e = P_{r(e)}\neq 0$ it follows that $S_e \neq 0$.
\end{proof}

As a consequence of the Corollary, we see that for every path $\alpha$ in $E$, the element $S_{\alpha} \in C^*(E)$ is nonzero, moreover, if $\alpha$ and $\beta$ are paths in $E$ such that $r(\alpha) = r(\beta)$ then $S_{\alpha}S_{\beta}^* \neq 0$.

\begin{cor}
Let $E = (E^0, E^1, r,s)$ be a row-countable graph. Then $E$ satisfies the condition $(L)$ if and only if every representation $\pi: C^*(E) \to B(L^2(X))$ induced from a branching system is faithful.
\end{cor}
\begin{proof}
For the direct implication, as $\pi(P_v)$ and $P_v$ are nonzero elements, the result follows from \cite[Theorem (2)]{uck}. For the converse, let's show that if $\alpha = e_1 \ldots e_n$ is a path without exit then exists a representation $\varphi$ such that $\varphi$ is not injective. Let $X$ be the branching system as in the proof of Theorem \ref{princ}. We redefine only the maps $f_{e_i}$ for $1\leq i\leq n$. From the proof of Theorem \ref{princ} we get $R_{e_i} = (0,1] \times \{e_i\}$ and  
$D_{r(e_i)} = (0,1] \times s^{-1}(r(e_i)) = (0,1] \times \{e_{i+1}\}$ for each $1\leq i\leq n-1$ and $D_{r(e_n)} = (0,1] \times s^{-1}(r(e_n)) = (0,1] \times \{e_1\}.$
For  $i = 1$, define $f_{e_1} : D_{r(e_1)} \to R_{e_1}$  by $f_{e_1}(t,e_2) = (\sqrt{t}, e_1)$ and for $i = n$, define $f_{e_n} : D_{r(e_n)} \to R_{e_n}$ by $f_{e_n}(t,e_1) = (\sqrt{t}, e_n)$. For each $i = 2, \ldots, n -1$ define $f_{e_i}: D_{r(e_i)} \to R_{e_i}$ by $f_{e_i}(t, e_{i + 1}) = (t, e_i)$. 
So we get a new $E$-branching system. For this branching system, by Theorems \ref{rep} we get a representation $\psi: C^*(E) \to B(L^2(X))$ such that $\psi(S_{e_1} \ldots S_{e_n}) \neq \psi( P_{s(e_1)})$. In particular, $S_{e_1} \ldots S_{e_n} \neq P_{s(e_1)}$.

Moreover,  we can choose $f_{e_1} = (t,e_2) = (\sqrt{t}, e_1)$, $f_{e_n}(t,e_1) = (t^2, e_n)$ and $f_{e_i}(t, e_{i + 1}) = (t,e_i)$ for each $1< i < n$, getting another $E$-branching system $Y$. Let $\varphi: C^*(E) \to B(L^2(Y))$ be the induced representation. For $\phi \in L^2(Y)$ it holds that

$$\varphi(S_{e_1} \ldots S_{e_n})(\phi) = \chi_{R_{e_1}} . ( \Phi_{f_{e_n}^{-1} \circ \ldots \circ f_{e_1}^{-1}} ) . (\phi \circ f_{e_n}^{-1} \circ \ldots \circ f_{e_1}^{-1}) = $$ $$=  \chi_{R_{e_1}} . \phi = \chi_{D_{r(e_n)}} . \phi = \chi_{D_{s(e_1)}} . \phi = \varphi(P_{s(e_1)})(\phi)$$
and then $\varphi(S_{e_1} \ldots S_{e_n}) = \varphi(P_{s(e_1)})$. Since $S_{e_1} \ldots S_{e_n} \neq P_{s(e_1)}$ then $\varphi$ is not injective.

\end{proof}

\section{Unitary equivalence and permutative representations}

In this section we show that each permutative  representation of graph C*-algebras in Hilbert spaces operators, even the spaces being non-separable,  are unitarily equivalent to representations induced from branching systems. Moreover, we find a class of graphs where each representation is permutative.

Let $\varphi: C^*(E) \to B(H)$ a representation of $C^*(E)$. The relations that define  the universal C*-algebra $C^*(E)$ and the fact that $\varphi$ is a $*$-homomorphism ensure that $\{\varphi(P_v)\}_{v \in E^0}$ and $\{\varphi(S_e) \varphi( S_e^*)\}_{e \in E^1}$ are families of mutually orthogonal projections. For each edge $e$ and vertex $v$ let $$H_v := \varphi(P_v)(H) \,\ \,\ \,\ \,\ H_e = \varphi(S_e)\varphi(S_e)^*(H).$$

As $\varphi(P_v)$ and $\varphi(S_e)\varphi(S_e)^*$ are projections, $H_v$ and $H_e$ are closed subspaces of $H$. Moreover, it holds that:

\begin{enumerate}
\item If $v \neq w$ then $H_v \cap H_w = \{0_H\}$,
\item If $e \neq f$ then $H_e \cap H_f = \{0_H\}$,
\item The restriction of $\varphi(S_e)$ given by $\varphi(S_e) : H_{r(e)} \to H_e$ is a surjective, isometric and unitary operator,
\item If $0 < \#s^{-1}(v) < \infty$ then $H_v = \bigoplus\limits_{e \in s^{-1}(v)}H_e$ and if  $\#s^{-1}(v) = \infty$ then we may write $H_v = \left( \bigoplus\limits_{e \in s^{-1}(v)}H_e \right) \bigoplus V_v$ where $V_v = \left( \bigoplus\limits_{e \in s^{-1}(v)}H_e \right)^{\perp}$.
\item We write $H = \left( \bigoplus\limits_{v \in E^0} H_v \right) \bigoplus V$ where $V = \left( \bigoplus\limits_{v \in E^0} H_v \right)^{\perp}$. 
\end{enumerate}

Recall that $\bigoplus\limits_{\lambda \in \Lambda}A_{\lambda} := \overline{span \left( \bigcup \limits_{\lambda \in \Lambda}A_{\lambda}\right)}$ provided that $\{A_{\lambda}\}_{\lambda \in \Lambda}$ is a family of mutually orthogonal subspaces of a Hilbert space (that's the case). The proof of properties above follows from the relations that define $C^*(E)$. Other interesting fact is the relation between total orthonormal sets in $A_{\lambda}$ with total orthonormal sets in $\bigoplus\limits_{\lambda \in \Lambda}A_{\lambda}$; this relation is given by the following proposition.

\begin{prop}
Let $X$ be a Hilbert space and $\{A_{\lambda} \}_{\lambda \in \Lambda}$ a collection of mutually orthogonal subspaces of $X$. If for each $\lambda \in \Lambda$, $B_{\lambda} \subseteq A_{\lambda}$ is a total orthonormal set in $A_{\lambda}$ then $\bigcup\limits_{\lambda \in \Lambda} B_{\lambda}$ is a total orthonormal set in $H := \bigoplus\limits_{\lambda \in \Lambda}A_{\lambda}$.
\end{prop}
\begin{proof}
It is easy to see that $\bigcup\limits_{\lambda \in \Lambda} B_{\lambda}$ is an orthonormal set.
We show that $\bigcup\limits_{\lambda \in \Lambda} B_{\lambda}$ is total in $H$.

Fix $h \in H$ and $\epsilon > 0$. As $H = \bigoplus\limits_{\lambda \in \Lambda}A_{\lambda}$, there exists $a = a_{\lambda_1} + \ldots + a_{\lambda_n} \in span \bigcup\limits_{\lambda \in \Lambda} A_{\lambda}$ (where $n \in \mathbb{N}$ and $\lambda_i \in A_{\lambda_i}$ for each $i=1,\ldots,n$) such that $\|h - a\| < \dfrac{\epsilon}{2}$. As $B_{\lambda_i}$ is a total orthonormal set in $A_{\lambda_i}$ then there exists $b_{\lambda_i} \in span \,\ B_{\lambda_i}$ such that $\|a_{\lambda_i} - b_{\lambda_i}\| < \dfrac{\epsilon}{2n}$. Define $b = b_{\lambda_1} +  \ldots + b_{\lambda_n}$, note that $b \in span\bigcup\limits_{\lambda \in \Lambda} B_{\lambda}$ and moreover
\begin{align*}
\|h - b\| & \leq    \|h - a\| + \|a - b\| \\
& <  \dfrac{\epsilon}{2} + \|a_{\lambda_1} - b_{\lambda_1} + \ldots + a_{\lambda_n} - b_{\lambda_n}\| \\
& \leq  \dfrac{\epsilon}{2} + \|a_{\lambda_1} - b_{\lambda_1}\| + \ldots + \|a_{\lambda_n} - b_{\lambda_n}\| \\
& <    \dfrac{\epsilon}{2} + n \dfrac{\epsilon}{2n} = \epsilon. \\
\end{align*}
\end{proof}

We show some consequences of these properties.
\begin{prop} Let $E$ be a graph and $\varphi: C^*(E) \to B(H)$ a representation.
\begin{enumerate}
\item For each $x \in C^*(E)$, $\varphi(x)$ vanishes at $V$ (where $V$ is as above).
\item If $H$ is separable then $\varphi(v) \neq 0$ for, at most, a countable number of vertices $v \in E^0$. 
\end{enumerate}
\end{prop}
\begin{proof}
First we prove 1. We know that $H = \left( \bigoplus\limits_{v \in E^0} H_v \right) \bigoplus V$. Let $y \in V$ and $\mu, \nu$ paths in $E$ such that $r(\mu) = r(\nu)$. If $a = S_{\mu}S_{\nu}^*$ then 
$$\|\varphi(a)(y)\|^2 = \langle \varphi(a)(y), \varphi(a)(y) \rangle = \langle y, \varphi(a^*a)(y) \rangle= $$ $$ = \langle y, \varphi(S_{\nu}S_{\mu}^*S_{\mu}S_{\nu}^*)(y) \rangle = \langle y, \varphi(P_{s(\nu)}) \varphi(S_{\nu}S_{\mu}^*	S_{\mu}S_{\nu}^*)(y) \rangle = 0.$$ 

As each element of $C^*(E)$ may be approximated by elements of the form $S_{\mu}S_{\nu}$ the result follows from the the continuity of the inner product in $H$.

Now we prove 2. If $\varphi(v) \neq 0$ for a uncountable number of vertices $v \in E^0$ we choose, for every $H_v$, a total orthonormal subset $B_v \subseteq H_v$. Then $\bigcup\limits_{v \in E^0}B_v$ is a uncountable total orthonormal set of $ \bigoplus\limits_{v \in E^0} H_v $. Since $H$ is separable, $\bigoplus\limits_{v \in E^0} H_v \subseteq H$ is separable. This is a contradiction because every total orthonormal subset of a separable Hilbert space is countable.
\end{proof}

\begin{defi}
Let $\varphi: C^*(E) \to B(H)$ a representation. We say that $\varphi$ is permutative if for every $e \in E^1$, $v \in E^0$ there exist total orthonormal sets $B_e \subseteq H_e$ and $B_v \subseteq H_v$ such that
\begin{itemize}
\item If $e \in s^{-1}(v)$ then $B_e \subseteq B_v$,
\item If $0 < \#s^{-1}(v) < \infty$ then $B_v = \bigcup\limits_{e \in s^{-1}(v)}B_e$,
\item $\varphi(S_e)(B_{r(e)}) = B_e$.
\end{itemize}
\end{defi}

As $\varphi(S_e): H_{r(e)} \to H_e$ is unitary, the third condition of the definition above is equivalent to $B_{r(e)} = \varphi(S_e)^*(B_e)$. Furthermore, $\varphi(S_e)(B_{r(e)})$ is always a total orthonormal set in $H_e$ because $\varphi(S_e):H_{r(e)}\rightarrow H_e$ is isometric and surjective.

\begin{exem} Here we show an example of a permutative representation. Let $E$ be the graph as follows.

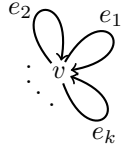
\begin{center}
\begin{tikzpicture}[->,auto,
                    thick]
\tikzset{every state/.style={minimum size=1pt}} 
\tikzset{every loop/.style={min distance=10mm,looseness=20}} 

\node[state,inner sep=0.5pt,draw=none] (A)   {$v$};
\node[state,inner sep=0.5pt,draw=none] (B) [below  left=0.1cm and 0.1cm of A]   {$.$};
\node[state,inner sep=0.5pt,draw=none] (C) [below right=0.025cm and 0.05cm of B]   {$.$};
\node[state,inner sep=0.5pt,draw=none] (D) [above left=0.025cm and 0.0005cm of B]   {$.$};
\node[state,inner sep=0.5pt,draw=none] (F) [above=0.02cm of D]   {$.$};

\path[->] (A) edge [in=10,out=60,loop] node[above] {$e_1$} (A);
\path[<-] (A) edge  [in=130,out=80,loop] node[left]{$e_2$} (A);
\path[<-] (A) edge  [in=290,out=340,loop] node[below]{$e_k$} (A);
\end{tikzpicture}
\end{center}
For each $i=1,\ldots,k$ define the operator $U_i : l^2 \to l^2$ given by
$$U_i \left( (x_n)_{n \in \mathbb{N}}\right) = (0, \ldots, 0, \overbrace{x_1}^{i}, 0, \ldots, 0, \overbrace{x_2}^{i + k}, 0, \ldots, 0, \overbrace{x_3}^{i + 2k}, \ldots),$$ and notice that $$U_i^*\left( (y_n)_{n \in \mathbb{N}}\right) = (y_i, y_{i + k}, y_{i + 2k}, \ldots ).$$

By the universal property of $C^*(E)$ there exists a $*$-homomorphism $\varphi: C^*(E) \to B(l^2)$ such that $\varphi(S_{e_i}) = U_i$ and $\varphi(P_v) = Id$. Let's show that $\varphi$ is permutative. As $$U_iU_i^*(y) = (0, \ldots, 0, \overbrace{y_i}^{i}, 0, \ldots, 0, \overbrace{y_{i + k}}^{i + k}, \ldots) )$$  then
 $$H_{e_i} = U_iU_i^*(l^2) = \overline{ span \{\delta_i, \delta_{i + k}, \delta_{i + 2k} \ldots \} }.$$
 where $(\delta_n)_{n \in \mathbb{N}}$ is the canonical basis of $l^2$. For each $i=1,\ldots,k$ choose $B_{e_i} = \{\delta_i, \delta_{i + k}, \delta_{i + 2k} \ldots \} \subseteq H_{e_i} = U_iU_i^*(l^2)$ and define:
$$B_v := \bigcup\limits_{e \in s^{-1}(v)} B_e = \bigcup\limits_{i = 1}^k B_{e_i} = \{\delta_1, \delta_2, \delta_3, \ldots \} = (\delta_n)_{n \in \mathbb{N}}.$$
Note that $\pi(S_{e_i})(B_v) = \pi(S_{e_i})( (\delta_n)_{n \in \mathbb{N}} ) = B_{e_i}$ because $\pi(S_{e_i})(\delta_1) = \delta_i$, $\pi(S_{e_i})(\delta_2) = \delta_{i + k}$, $\pi(S_{e_i})(\delta_3) = \delta_{i + 2k}$ and so on. So $\varphi$ is permutative. 
\end{exem}

\begin{exem}\label{exemploloop}
Now we will show an example of a non-permutative representation. Let $E$ be the graph 
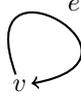
\begin{center}
\begin{tikzpicture}[->,auto,node distance=2.8cm,
                    thick]
\tikzset{every state/.style={minimum size=1pt}} 
\tikzset{every loop/.style={min distance=10mm,looseness=20}} 

\node[state,inner sep=0.5pt,draw=none] (A)   {$v$};

\path[->] (A) edge [in=10,out=110,loop] node[auto] {$e$} (A); 
    \end{tikzpicture}    
  \end{center}
Define $U: \mathbb{C}^2 \to \mathbb{C}^2$ by $U(z,w) = (iz,iw)$. By the universal property of $C^*(E)$ there exists a $*$-homomorphism $\varphi: C^*(E) \to  B(\mathbb{C}^2)$ such that $\varphi(S_e) = U$ and  $\varphi(P_v) = Id$. Suppose $\varphi$ permutative.  Then there exist an orthonormal total set $B_e \subseteq H_e = UU^*(\mathbb{C}^2) = \mathbb{C}^2$ and as orthonormal total set $B_v \subseteq H_v = \mathbb{C}^2$ such that $B_v = B_e$ and $\varphi(S_e)(B_{r(e)}) = U(B_{r(e)}) = B_e$. As $B_e$ is an orthonormal total set in $\mathbb{C}^2$ we can write $B_e = \{ (z_1, z_2), (w_1, w_2) \}$. Thus
$$\varphi(S_e)(B_{r(e)}) = U(B_{r(e)}) = \{ (iz_1, iz_2), (iw_1, iw_2) \} \neq B_e.$$
That's a contradiction. So $\varphi$ isn't permutative.
\end{exem}

Now we prove the main theorem of this section. This theorem is a generalization of Theorem 2.1 of \cite{B}. 

\begin{theorem}\label{bpb}
Let $\varphi: C^*(E) \to B(H)$ be a representation and suppose that $\varphi$ is permutative. Then there exists a representation $\pi: C^*(E) \to B(l^2({\Lambda}))$, arising from a branching system, such that $\varphi$ and $\pi$ are unitarily equivalent.
\end{theorem}

\begin{proof} We use the same notations as in the beginning of this section. Recall that $H = \left( \bigoplus\limits_{v \in E^0} H_v \right) \bigoplus V$. Choose $B_v \subseteq H_v$ a total orthonormal set. Moreover, defining $B' = \bigcup\limits_{v \in E^0}B_v$ we obtain a total orthonormal set in $\bigoplus\limits_{v \in E^0} H_v$. Also, choosing $D$ a orthonormal total set in $V$ we define $B = B' \cup D$ and write $B = \{b_{\lambda}\}_{\lambda \in \Lambda}$, which is a total orthonormal set in $H$.	Let  $(\Lambda, \eta)$ be the measure space where $\eta$ is the counting measure. For each $e \in E^1$, $v \in E^0$ define the sets
$$R_e = \{\lambda \in \Lambda \,\ | \,\ b_{\lambda} \in B_e \} \hspace{1cm} \text{ and } \hspace{1 cm} D_v = \{ \lambda \in \Lambda \,\ | \,\ b_{\lambda} \in B_v \}.$$
Now we define the desired branching system. If $e \neq f$, $B_e \cap B_f = \emptyset$ (because $H_e \cap H_f = \{0_H\}$) and then $R_e \cap R_f = \emptyset$. The same argument shows that $D_v \cap D_w = \emptyset$ (for  $v \neq w$). 
For an edge $e$ and $\lambda \in R_e$, as $e \in s^{-1}(s(e))$ and $\varphi$ is permutative then $B_e \subseteq B_{s(e)}$, and so $b_{\lambda} \in B_{s(e)}$. This means that $\lambda \in D_{s(e)}$ and so $R_e \subseteq D_{s(e)}$. Moreover, if $v \in E^0$ is such that $0 < \#s^{-1}(v) < \infty$ then $B_v = \bigcup\limits_{e \in s^{-1}(v)}B_e$ and so
$$\lambda \in D_v  \Leftrightarrow b_{\lambda} \in B_v \Leftrightarrow b_{\lambda} \in B_{e'} \textrm{ for some } e'\in s^{-1}(v) \Leftrightarrow \lambda \in \bigcup\limits_{e \in s^{-1}(v)}R_e.$$
So $D_v = \bigcup\limits_{e \in s^{-1}(v)}R_e$. For ech edge $e$ we define the function $f_e: D_{r(e)} \to R_e$ by the following rule: given $\lambda_0 \in D_{r(e)}$ we have $b_{\lambda_0} \in B_{r(e)}$ and as $\varphi$ is permutative  $\varphi(S_e)(b_{\lambda_0}) \in B_e$. Therefore $\varphi(S_e)(b_{\lambda_0}) = b_{\mu_0}$ for some $\mu_0 \in R_e$. Define $f_e(\lambda_0) = \mu_0$.

As $\varphi$ is permutative, $f_e$ is surjective and since $\varphi(S_e): H_{r(e)} \to H_e$ is injective so is $f_e$; we choose $f_e^{-1}$ as the inverse of $f_e$. It's clear that $f_e$ and $f_e^{-1}$ are both measurable functions. Moreover, as $\eta$ is the counting measure and $f_e$, $f_e^{-1}$ are bijections we have $\Phi_{f_e} = 1 = \Phi_{f_e^{-1}}$. Then $(\Lambda, \eta)$ is a branching system. 

By Theorem \ref{rep} there exists a representation $\pi: C^*(E) \to B(L^2(\Lambda, \eta)) = B(l^2(\Lambda))$ such that
$$\pi(S_e)(\phi) = \chi_{R_e} . \Phi_{f_e^{-1}}^{\frac{1}{2}} . ( \phi \circ f_e^{-1} ) = \chi_{R_e} . ( \phi \circ f_e^{-1} ) \hspace{1cm}
\textrm{and}  \hspace{1cm} \pi(P_v)(\phi) = \chi_{D_v} . \phi.$$

We define $U: span  \{ b_{\lambda} \}_{\lambda \in \Lambda}  \to l^2(\Lambda)$ given by $$U(\alpha_1 b_{\lambda_1} + \ldots + \alpha_n b_{\lambda_n} ) = \alpha_1 \chi_{ \{ \lambda_1 \}} + \ldots + \alpha_n \chi_{ \{ \lambda_n \}}.$$
Its clear that $U$ is linear and as $\{b_{\lambda}\}_{\lambda \in \Lambda}$ is a orthonormal set in $H$ and $\{ \chi_{ \{\lambda\}} \}_{\lambda \in \Lambda}$ is a orthonormal set in $l^2(\Lambda)$ then $\|x\|^2 = \sum\limits_{i = 1}^{n}|\alpha_i|^2 = \|U(x)\|^2$. This shows that $U$ is an isometric operator. 
Now we can extend the operator $U$ to a operator (which we also call $U$) from $H = \overline{span ( \{ b_{\lambda} \}_{\lambda \in \Lambda} )}$ to $l^2(\Lambda)$. 
It is easy to see that $U$ is a unitary operator.




\textbf{Claim 1.} $U^* \pi(S_e) U = \varphi(S_e)$ for each $e \in E^1$.

For $b_{\lambda_0} \in B$ it holds that $\pi(S_e)U(b_{\lambda_0}) = \pi(S_e)(\chi_{ \{\lambda_0\} }) = \chi_{R_e} ( \chi_{ \{\lambda_0 \} } \circ f_e^{-1} )$, and so
$$\pi(S_e)U(b_{\lambda_0})(\mu) = \begin{cases}
1, \textrm{ if } \mu \in R_e \textrm{ and } \mu = f_e(\lambda_0).\\
0, \textrm{ otherwise.}\\
\end{cases}$$

Suppose  $b_{\lambda_0} \notin B_{r(e)}$. In this situation $\lambda_0 \notin D_{r(e)}$. If $\mu \in \Lambda$ is such that $\mu \in R_e$ e $\mu = f_e(\lambda_0)$ then $f_e^{-1}(\mu) = \lambda_0 \in D_{r(e)}$, that's a contradiction. Then $\pi(S_e)U(b_{\lambda_0}) = 0$ and $U^*\pi(S_e)U(b_{\lambda_0}) = 0$. Furthermore, $\varphi(S_e)(b_{\lambda_0}) = 0$ because if $b_{\lambda_0} \notin B_{r(e)}$, $b_{\lambda_0} \in B_v$ for some $v \neq r(e)$ or $b_{\lambda_0} \in D$. In the first case $b_{\lambda_0} \in H_v$ and then $b_{\lambda_0} = \varphi(P_v)(b_{\lambda_0})$. So,
$$\varphi(S_e)(b_{\lambda_0}) = \varphi(S_e)\varphi(P_{r(e)}) \varphi(P_v)(b_{\lambda_0}) = \varphi(S_e)\varphi(P_{r(e)} P_v)(b_{\lambda_0}) = 0.$$
In the second case $b_{\lambda_0} \in D \subseteq \left( \bigoplus\limits_{v \in E^0} H_v \right)^{\perp} \subseteq H_{r(e)}^{\perp} = Ker \,\ \varphi(P_{r(e)})$. It follows that $\varphi(S_e)(b_{\lambda_0}) = \varphi(S_e)\varphi(P_{r(e)})(b_{\lambda_0}) = 0$. In both situations it holds that
$$U^*\pi(S_e)U(b_{\lambda_0}) = 0 = \varphi(S_e)(b_{\lambda_0}).$$

Now let $b_{\lambda_0} \in B_{r(e)}$. In this case as $\varphi$ is permutative then $\varphi(S_e)(b_{\lambda_0}) = b_{\mu_0}$ for some $b_{\mu_0} \in B_e$, so $f_e(\lambda_0) = \mu_0$. Then,
$$\pi(S_e)U(b_{\lambda_0}) = \begin{cases}
1, \textrm{ if } \mu \in R_e \textrm{ and } \mu = f_e(\lambda_0).\\
0, \textrm{ otherwise. }
\end{cases}
=
\begin{cases}
1, \textrm{ if } \mu = \mu_0.\\
0, \textrm{ otherwise. }
\end{cases}
= \chi_{ \{ \mu_0 \}}.$$
Therefore,
$$U^*\pi(S_e)U(b_{\lambda_0}) = U^*(\chi_{ \{ \mu_0 \}}) = b_{\mu_0} = \varphi(S_e)(b_{\lambda_0}).$$
By linearity and continuity follows Claim 1.

\textbf{Claim 2.} $U^* \pi(P_v) U = \varphi(P_v)$ for each $v \in E^0$.

Let $b_{\lambda_0} \in B$. Then $\pi(P_v)U(b_{\lambda_0}) = \chi_{D_v} \chi_{ \{ \lambda \} }$, and so
$$U^*\pi(P_v)U(b_{\lambda_0}) = U^*(\chi_{D_v} \chi_{ \{ \lambda_0 \} }) = U^*(\chi_{ D_v \cap \{ \lambda_0 \} } ) = \chi_{B_v} (U^* (\chi_{ \{ \lambda_0 \} }) ) .$$
Moreover, $\varphi(P_v)(b_{\lambda_0})  = \chi_{B_v}(b_{\lambda_0})$ and then
$$ U^*\pi(P_v)U(b_{\lambda_0}) = \chi_{B_v} (U^*(\chi_{ \{ \lambda_0 \} })) = \chi_{B_v} (b_{\lambda_0}) = \varphi(P_v)(b_{\lambda_0}).$$
By linearity and continuity it holds that for every $h \in H$, and so $U^*\pi(P_v)U(h) = \varphi(P_v)(h)$, a so Claim 2 is proved.

Finally, since $\pi$ and $\varphi$ are homomorphisms then for every $x \in C^*(E)$ it holds that $U^*\pi(x)U = \varphi(x)$.
\end{proof}

For separable Hilbert spaces $H$, since each orthonormal total set of such spaces are countable, we get from the previous theorem the following corollary.

\begin{cor}[2.1,\cite{B}]
Let $\varphi: C^*(E) \to B(H)$ be a representation, suppose that $H$ is separable and $\varphi$ is permutative. Then exists a representation $\pi: C^*(E) \to B(l^2(\mathbb{N}))$, induced by a branching system, such that $\varphi$ and $\pi$ are unitarily equivalent.
\end{cor}

\section{Graphs whose all representations are permutative}

In this section we prove that permutative representations are unitarily equivalent to representations induced by branching systems. Now, our goal is to find a class of graphs such that every representation is permutative. For this, we need some language that was developed in \cite{B}.

\begin{defi}[3.1, 3.3: \cite{B}]
Let $E$ be a graph.
\begin{enumerate}
\item For $f \in E^1$ and $v \in E^0$ we say that $f$ and $v$ are adjacent if $r(f) = v$ or $s(f) = v$.
\item For $u,v \in E^0$ such that $u \neq v$, we say that $u$ and $v$ are adjacent if exists a edge $f \in E^1$ such that $f$ to $u$ and $f$ is adjacent to $v$.
\item For $f,g \in E^1$ with $f \neq g$, we say that $f,g$ are adjacent if there exists a vertex $v \in E^0$ such that $v$ is adjacent to $f$ and $v$ is adjacent to $g$.
\item A path between $u,v \in E^0$ is a pair $(u_0u_1 \ldots u_n, e_1e_2 \ldots e_n)$; where $u_i \in E^0$ for each $i = 0,\ldots, n$, $e_i \in E^1$ for each $i= 1, \ldots ,n$;  and:
\begin{enumerate}
\item $u_0 = u$ and $u_n = v$,
\item $e_i \neq e_j$ if $i \neq j$,
\item for each $i=1,\ldots,n$ it holds that $r(e_i) = u_{i-1}$ and $ s(e_i) = u_i$ or $r(e_i) = u_i$ and $ s(e_i) = u_{i-1}$.
\end{enumerate}
\item A cycle is a path $(u_0u_1 \ldots u_n, e_1e_2 \ldots e_n)$ such that $u_0 = u_n$.
\item We say that $E$ is $P$-simple if $E$ has not loops and for every $v,u \in r(E^1) \cup s(E^1)$ with $v \neq u$ there exists at most one path between $v$ and $u$.
\item We say that a vertex $v$ in $E$ is a extreme vertex of $E$ if $s^{-1}(v) \cup r^{-1}(v) = 1$ and $v$ is not a basis of a loop. In this case, the only edge adjacent to $v$ is called an extreme edge.
\end{enumerate}
\end{defi}

The reader can check that $E$ is $P$-simple if, and only if, $E$ has no cycles. 

Let $E$ be a graph. We define $\mathcal{E}_1$ as being the subgraph $\mathcal{E}_1 = (E^0 - X_1, E^1 - Y_1, r_1, s_1)$ of $E$ where $X_1$ is the set of extreme vertices of $E$, $Y_1$ is the set of extreme edges of $E$ and $r_1$, $s_1$ denote the restrictions of $r$ and $s$ to $E^1 - Y_1$. The vertices in $X_1$ are called the level $1$ vertices of $E$ and the edges in $Y_1$ are called the level $1$ edges of $E$.

More generally, for each $i \in \mathbb{N}$ we define $\mathcal{E}_i$ as the subgraph $\mathcal{E}_i = (E^0 - X_1 \cup \ldots \cup X_i, E^1 - Y_1 \cup \ldots \cup Y_i, r_i, s_i)$ of $\mathcal{E}_{i - 1}$ where $X_i$ is the set of extreme vertices of $\mathcal{E}_{i - 1}$, $Y_i$ is the set of extreme edges of $\mathcal{E}_{i - 1}$ and $r_i$, $s_i$ denote the restrictions of $r$ and $s$ to $E^1 - Y_1 \cup \ldots \cup Y_i$.  The vertices in $X_i$ are called the level $i$ vertices of $E$ and the edges in $Y_i$ are called the level $i$ edges of $E$.

Note that the extreme vertices (edges) of $\mathcal{E}_{i - 1}$ are exactly the level $i - 1$ vertices (edges) of $E$.

\begin{defi}[3.2: \cite{B}]
 Let $E = (E^0, E^1, r,s)$ be a graph and $V \subseteq E^0$. We say that $V$ is conneceted in $E$ if for all $u,v \in V$ there exists a path (in $E$) between $u$ e $v$. If $V = E^0$ is connected in $E$ we say that $E$ is connected.
\end{defi}

Fix a graph $E$. Given $u, v \in r(E^1) \cup s(E^1)$ we say that $u \sim v$ if $u = v$ or there exists a path between $u$ and $v$. Note that $\sim$ is a equivalence relation in $Z :=  r(E^1) \cup s(E^1)$. Let $\Delta$ be a set with exactly one member of each equivalence class. Then we can write $Z = \bigcup\limits_{v_i \in \Delta}^{.}Z_{v_i}$ where $Z_{v_i}$ denote the equivalence class of $v_i$. As consequence if $R = E^0 - Z$ then $E^0 = \bigcup\limits_{v_i \in \Delta}^{.}Z_{v_i} \bigcup\limits^{.} R$. The elements of $R$ are called isolated vertices. It's easy to see that
$$(s^{-1}(Z_v), Z_v,  s_{|s^{-1}(Z_v)}, r_{|s^{-1}(Z_v)}) = (r^{-1}(Z_v), Z_v,  s_{|r^{-1}(Z_v)}, r_{|r^{-1}(Z_v)})$$ is a subgraph of $E$.

The following proposition are very useful in the next results.

\begin{prop}\label{auxiliar}
Let $E$ be a graph. Then
\begin{enumerate}
    \item[a)] Suppose that $\mathcal{E}_i$ is defined for $i \in \mathbb{N}$. If $Z := r(E^1) \cup s(E^1)$ is connected in $E$ then $Z - X_1 \cup \ldots \cup X_i$ is connected in $\mathcal{E}_i$.
    \item[b)] [3.4: \cite{B}] If $v \in X_n$ for some $n \in \mathbb{N}$ then exists at most one vertex $w$ in $E$ such that $w$ is adjacent to $v$ and the level of $w$ is greater or equal to $n$.
    \item[c)] [3.4: \cite{B}]If $Z = \bigcup\limits_{i=1}^{m}X_i$, for some $m \in \mathbb{N}$ and $Z$ is connected then
\begin{enumerate}
\item [I)] Given $v \in X_n$ with $n < m$ exists one, and only one vertex $w$ with level greater than $n$ such that $w$ is adjacent to $v$.
\item [II)] The set $X_m$ has exactly two vertices and exists exactly one edge adjacent to the two vertices.
\end{enumerate}

\item[d)] [3.4: \cite{B}] If $Z = ( \bigcup\limits_{i=1}^{m}X_i ) \bigcup\limits^{.} \{  \overline v\} $ for each $m \in \mathbb{N}$ and $Z$ is connected then:
\begin{enumerate}
\item [I)] If $v \in X_n$ and $n < m$ then $v$ is adjacent to $\overline{v}$ or exists exactly one vertex $w$ with level greater than $n$ such that $w$ is adjacent to $v$.
\item [II)] For each $v \in X_m$ exists exactly one edge adjacent to $v$ and $\overline{v}$.
\end{enumerate}
    
\end{enumerate}
\end{prop}

\begin{proof}
We only will prove a). The proofs of the other items can be found in \cite{B}. Suppose that $i = 1$. We will prove that $Z - X_1$ is connected in $\mathcal{E}_1$. Let $u,v \in Z - X_1$. As $Z$ is connected in $E$ then there exists a path (in $E$), $(u_0 \ldots u_p, e_1 \ldots e_p)$, between $u$ and $v$. Notice that $ \# s^{-1}(u_i) \cup r^{-1}(u_i) \geq 2$  for each $i = 2, \ldots, p - 1$. Then $u_i$ isn't a extreme vertex of $E$; therefore $u_i \in Z - X_1$ for each $i=1,\ldots,p$. Moreover, if $e_i \notin E^1 - Y_1$ then $e_i \in Y_1$ and $e_i$ is a extreme edge of $E$. Thus $s(e_i)$ or $r(e_i)$ are extreme edges of $E$, that's a contradiction. Therefore $e_i \in E^1 - Y_1$ e $(u_0 \ldots u_p, e_1 \ldots e_p)$ is a path in $\mathcal{E}_1$ between $u$ and $v$. Now the proof follows by inductive arguments.
\end{proof}

For more details about adjacency, extreme vertices and connected graphs we recommend \cite{B}. The next theorem is a generalization of [3.4 d), \cite{B}]. 

\begin{theorem}\label{ppp}
Let $E$ be a $P$-simple graph. If $Z := r(E^1) \cup s(E^1)$ is connected and there exists $n \in \mathbb{N}$ such that $\mathcal{E}_{n}$ exists and has finitely many vertices then $Z = \bigcup\limits_{i=1}^{m}X_i$ or $Z = ( \bigcup\limits_{i=1}^{m}X_i ) \bigcup\limits^{.} \{  \overline v\} $ for some $m \in \mathbb{N}$.
\end{theorem}

\begin{proof}
Let $n_0$ the smallest number such that $\mathcal{E}_{n_0}$ has finite vertices. Let $m$ the biggest number such that $\mathcal{E}_m$ is defined. Suppose that $Z := r(E^1) \cup s(E^1) \neq \bigcup\limits_{i = 1}^{m}X_i$. As $\bigcup\limits_{i = 1}^{m}X_i \subseteq Z$ then there exists $\overline{v} \in Z - \bigcup\limits_{i = 1}^{m}X_i$. So $\overline{v}$ is a vertex of $\mathcal{E}_{m}$ and we can suppose that $\mathcal{E}_{m}$ has $N$ ($N \in \mathbb{N}$) vertices.

We claim that $\overline{v}$ is a isolated vertex of the graph $\mathcal{E}_m$. Otherwise, there exists an edge $e_1$ in the graph $\mathcal{E}_m$ such that $e_1$ is adjacent to $\overline{v}$. Suppose that $r(e_1) = \overline{v}$. Of course $s(e_1) \neq \overline{v}$ because $E$ is $P$-simple. Let $v_0 = s(e_1)$ and as $\mathcal{E}_m$ is a subgraph, $v_0$ is a vertex of $\mathcal{E}_m$.

As $\overline{v}$ is not a extreme vertex of $\mathcal{E}_m$ then $\# s^{-1}(\overline{v}) \cup r^{-1}(\overline{v}) \geq 2$. Let $e_2$ be another edge adjacent to $\overline{v}$ in $\mathcal{E}_m$. Without loss of generality, we assume $s(e_2) = \overline{v}$ and $r(e_2) = v_2$. Note that $v_2 \neq v_0$, $v_2 \neq \overline{v}$. Proceeding inductively we get a contradiction because $\mathcal{E}_m$ has finite vertices.

\begin{center}
\begin{tikzpicture}[->,auto,node distance=2.5cm,
                   thick]
\tikzset{every state/.style={minimum size=0pt}} 
\tikzset{every loop/.style={min distance=10mm,in=0,out=80,looseness=20}} 

\node[state,inner sep=0.5pt,draw=none] (A)   {$v_0$};
\node[state,inner sep=0.5pt,draw=none] (B) [right of =A]   {$\overline{v}$};
\node[state,inner sep=0.5pt,draw=none] (C) [right of =B]   {$v_2$};
\node[state,inner sep=0.5pt,draw=none] (D) [right of =C]   {$v_3 ...$};

\path (A) edge  node[above]{$e_1$} (B)
(B) edge  node[above]{$e_2$} (C)
(C) edge  node[above]{$e_3$} (D);
\end{tikzpicture}
\end{center}

Finally, suppose that $Z - \bigcup\limits_{i = 1}^{m}X_i$ has two or more elements. By the claim above the vertices are isolated in the graph $\mathcal{E}_m$ but this contradicts the fact that $\mathcal{E}_m$ is connected. Therefore $Z - \bigcup\limits_{i = 1}^{m}X_i$ has exactly one element.

\end{proof}

\begin{theorem}\label{gbpb}
Let $E = (E^0, E^1, r,s)$ a $P$-simple graph. Suppose that $Z := r(E^1) \cup s(E^1)$ is connected and suppose that there exists $n \in \mathbb{N}$ such that the graph $\mathcal{E}_{n}$ exists and has finitely many vertices. If $\varphi: C^*(E) \to B(H)$ is a representation then $\forall v \in E^0$ and $\forall e \in E^1$ then there exists total orthonormal sets $B_v$ and $B_e$ from $H_v$ and $H_e$, respectively, such that:
\begin{enumerate}
\item If $e \in s^{-1}(v)$ then $B_e \subseteq B_v$ and if $0 < \# s^{-1}(v) < \infty$ then $B_v = \bigcup\limits_{e \in s^{-1}(v)} B_e$.
\item If $e \in r^{-1}(v)$ then $\varphi(S_e)(B_v) = B_e$.
\end{enumerate}
\end{theorem}
\begin{proof}
The proof of this theorem consists in two inductive processes in the level of the vertices. By the previous theorem, $Z = \bigcup\limits_{i = 1}^{m}X_i$ or $Z = \bigcup\limits_{i = 1}^{m}X_i \cup \{\overline{v}\}$. Suppose that $Z = \bigcup\limits_{i = 1}^{m}X_i$.

\textbf{Step 1.} For each $v \in X_1^{VF}$ we choose a total orthonormal set $B_v \subseteq H_v$. For each 
$e \in r^{-1}(v)$ we define $B_e := \pi(S_e)(B_v)$. For all the other vertices and edges we choose arbitrary orthonormal total subsets $B_v \subseteq H_v$ and $B_e \subseteq H_e$. Then the conditions 1 and 2 are satisfied for all $v \in X_1^{VF}$ and for all $e \in E^1$ such that $r(e) \in X_1^{VF}$.

\textbf{Step 2.} Let $v \in X_2^{VF}$. If $s^{-1}(v) = \emptyset$ we choose $\overline{B}_v := B_v$, if $0 < \# s^{-1}(v) < \infty$ we choose $\overline{B}_v := \bigcup\limits_{e \in s^{-1}(v)} B_e$ and if $\# s^{-1}(v) = \infty$ we define $\overline{B}_v$ as a total orthonormal set of $H_v$ such that $B_e \subseteq \overline{B}_v$ for each $e \in s^{-1}(v)$. Thus, if $e \in r^{-1}(v)$ define $\overline{B}_e := \pi(S_e)(\overline{B}_v)$. For the other vertices and edges define $\overline{B_v} := B_v$ e $\overline{B_e} := B_e$.

\textbf{Claim 2.}: If $e \in r^{-1}(v)$ and $v \in X_2^{VF}$ then $r(e) \notin X_1^{VF}$ and $s(e) \notin X_2^{VF}$. 

Of course $r(e) \notin X_1^{VF}$. If $s(e) \in X_2^{VF}$ then $s(e)$ is a vertex of the graph $\mathcal{E}_1$ as well as $r(e)$; so $e$ is a edge of this graph. As $s(e) \in X_2^{VF}$ then there exists a vertex $w$ with level greater than $2$ and a edge $f$ in the graph $\mathcal{E}_1$ such that $r(f) = s(e)$. Note that $e \neq f$. Then $s(e)$ is a adjacent to $e$ and $f$ in the graph $\mathcal{E}_1$. That is a contradiction because $s(e)$ is a extreme vertex of $\mathcal{E}_1$. Therefore $s(e) \notin X_2^{VF}$.

The claim above ensures that the \textbf{Step 2} doesn't modifies the previous choices. Thus, after the \textbf{Step 2} we get total orthonormal sets $\overline{B_v}$ ($v \in E^0$) e $\overline{B_e}$ ($e \in E^1)$ such that 1 and 2 are satisfied for all $v \in X_1^{VF} \cup X_2^{VF}$.

We proceed inductively until the step $m - 1$. So, we get total orthonormal sets $B_v$ e $B_e$ such that $1$ and $2$ are satisfied for all $v \in X_1^{VF} \cup \ldots \cup X_{m - 1}^{VF}$.

\textbf{Step m.} Let $v \in X_m^{VF}$. If $s^{-1}(v) = \emptyset$ we define $\overline{B_v} := B_v$, if $0 < \# s^{-1}(v) < \infty$ we define $\overline{B_v} := \bigcup\limits_{e \in s^{-1}(v)} {B_e}$ and if $\# s^{-1}(v) = \infty$ we choose $\overline{B_v}$ a orthonormal total set of $H_v$ such that $B_e \subseteq \overline{B_v}$ for each $e \in s^{-1}(v)$. If $e \in r^{-1}(v)$ define $\overline{B_e} := \pi(S_e)(\overline{B_v})$. For the all the other vertices and edges we define $\overline{B_v} := B_v$ e $\overline{B_e} := B_e$.

\textbf{Claim m.} If $e \in r^{-1}(v)$ and $v \in X_m^{VF}$ then $r(e) \notin X_1^{VF} \cup \ldots \cup X_{m - 1}^{VF}$ and $s(e) \notin X_2^{VF} \cup \ldots \cup X_{m}^{VF}$. 

Of course $r(e) \not\in X_1^{VF} \cup \ldots \cup X_{m - 1}^{VF}$. Suppose that $s(e) \in X_i^{VF}$ for some $i =2, \ldots, m$. If $2 \leq i \leq m - 1$ the argument is the same as in \textbf{Claim 2.} If $i = m$ then $s(e) \in X_m^{VF}$ as well as $r(e)$, thus $r(e)$, $s(e)$ are distinct vertices of $\mathcal{E}_{m - 1}$ and $e$ is a edge of this graph. As $s(e)$ is a final vertex of $X_m$ exists a edge $f$ in this graph such that $r(f) = s(e)$. Of course $f \neq e$. As $X_m$ has exactly two elements (Proposition \ref{auxiliar}) then $s(f) = s(e)$ or $s(f) = r(e)$. In both cases we get a cycle, which is impossible, since $E$ is $P$-simple.

After this step we get total orthonormal sets $\overline{B_v}$ and $\overline{B_e}$ such that 1 and 2 are satisfied for all $v \in X_1^{VF} \cup \ldots \cup X_m^{VF}$.

\textbf{Step m + 1.} Let $v \in X_m^{VI}$. If $0 < \# s^{-1}(v) < \infty$ we define $\overset{\sim}{B_v} := \bigcup\limits_{e \in s^{-1}(v)}{\overline{B_e}}$ and if $\# s^{-1}(v) = \infty$ we choose $\overset{\sim}{B_v}$ a orthonormal total set of $H_v$ such that $\overline{B_e} \subseteq \overset{\sim}{B_v}$ for each $e \in s^{-1}(v)$. If $e \in r^{-1}(v)$ define $\overset{\sim
}{B_e} = \pi(S_e)(\overset{\sim}{B_v})$. For all the other vertices and edges define $\overset{\sim}{B_v} := \overline{B_v}$ e $\overset{\sim}{B_e} := \overline{B_e}$.

With analogue arguments as used above we conclude that this construction don't change our previous choices. Now, with inductive arguments we conclude the proof in the case $Z = \bigcup\limits_{i = 1}^{m}X_i$. 

Now, suppose that $Z = \bigcup\limits_{i = 1}^{m}X_i \bigcup\limits^{.} \overline{v}$. The steps $1$ until $m$ are the same. After this, we need a extra step as follows:

\textbf{Extra Step.}
We need deal with $\overline{v}$ in a extra step because $\overline{v}$ has not level. The argument is similar as in the other steps. If $s^{-1}(\overline{v}) = \emptyset$ define $\overset{\smile}{{B}_{\overline{v}}} := \overline{B}_{\overline{v}}$; if $0 < \#s^{-1}(\overline{v}) < \infty$ define $\overset{\smile}{{B}_{\overline{v}}} := \bigcup\limits_{e \in s^{-1}(\overline{v})} \overline{B}_e$ and if $\#s^{-1}(\overline{v}) = \infty$ we choose $\overset{\smile}{{B}_{\overline{v}}}$ such that $\overline{B}_e \subseteq \overset{\smile}{{B}_{\overline{v}}}$ for every $e \in s^{-1}(\overline{v})$. For each $e \in r^{-1}(\overline{v})$ define $\overset{\smile}{B_e} := \varphi(S_e)(\overset{\smile}{B_{\overline{v}}})$ and for the others vertices and edges define $\overset{\smile}{B_v} := \overline{B_v}$ and $\overset{\smile}{B_e} := \overline{B_e}$.

\textbf{Extra Claim.} If $e \in r^{-1}(\overline{v})$ then $r(e) \notin X_1^{VF} \cup \ldots \cup X_m^{VF}$ and $s(e) \notin X_2^{VF} \cup \ldots \cup X_m^{VF}$.

Of course $r(e) \notin X_1^{VF} \cup \ldots \cup X_m^{VF}$ and if $s(e) \in X_2^{VF} \cup \ldots \cup X_{m - 1}^{VF}$ we repeat the previous arguments (that's possible because $\overline{v} \in \mathcal{E}_i$ for each $i = 1, \ldots, m - 1$). If $s(e) \in X_m^{VF}$ then there exists an edge $f$ in the graph $\mathcal{E}_{m - 1}$ such that $r(f) = s(e)$ and $s(f) = \overline{v}$, then $(\overline{v} s(e) s(f), ef)$ is a cycle. That's a contradiction.

The next steps are the same as in the case $Z = \bigcup\limits_{i = 1}^{m}X_i$. So we conclude the case $Z = \bigcup\limits_{i = 1}^{m}X_i \bigcup\limits^{.}\{ \overline{v} \}$ and the proof of the theorem.
\end{proof}

\begin{cor}\label{everybpb}
Let $E = (E^0, E^1, r,s)$ be $P$-simple graph. Suppose that $Z := r(E^1) \cup s(E^1)$ is connected and exists $n \in \mathbb{N}$ such that the graph $\mathcal{E}_{n}$ exits and has finitely many vertices. Then every representation $\varphi: C^*(E) \to B(H)$ is unitarily equivalent to a representation arising from a $E$-branching system. 
\end{cor}

\begin{proof}
It follows from Theorems \ref{gbpb} e \ref{bpb}.
\end{proof}

\begin{exem}
Due to the previous corollary every representation of the two graphs below is unitary equivalent to a representation arising from a branching system

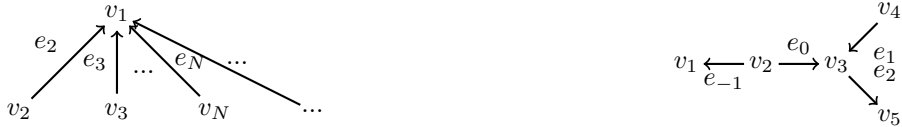
\begin{figure}[!htb]
\centering
\begin{minipage}{0.5\textwidth}
        \centering
        \begin{tikzpicture}[->,auto,node distance=1.3 cm,
                    thick]
\tikzset{every state/.style={minimum size=0pt}} 

\node[state,inner sep=0.5pt,draw=none] (A){$v_1$};
\node[state,inner sep=0.5pt,draw=none] (B)[below of =A] {$v_3 $};
\node[state,inner sep=0.5pt,draw=none] (C)[left of =B] {$v_2$};
\node[state,inner sep=0.5pt,draw=none] (D)[right of =B] {$v_{N}$};
\node[state,inner sep=0.5pt,draw=none] (E)[right of =D] {$...$};

\path 
(B) edge  node{$e_3$} (A)
(C) edge  node{$e_2$} (A)
(D) edge  node{$...$}node[right]{$e_{N}$} (A)
(E) edge  node[right]{$...$}  (A);

\end{tikzpicture}
    
    \end{minipage}%
    \begin{minipage}{0.5\textwidth}
        \centering
        \begin{tikzpicture}[->,auto,node distance=1cm,
                    thick]
\tikzset{every state/.style={minimum size=0pt}} 
\tikzset{every loop/.style={min distance=10mm,in=0,out=80,looseness=20}} 

\node[state,inner sep=0.5pt,draw=none] (A)   {$v_2$};
\node[state,inner sep=0.5pt,draw=none] (B) [right of =A]   {$v_3$};
\node[state,inner sep=0.5pt,draw=none] (C) [above right of =B]   {$v_4$};
\node[state,inner sep=0.5pt,draw=none] (D) [below right of =B]   {$v_5$};
\node[state,inner sep=0.5pt,draw=none] (F) [left of = A]   {$v_1$};

\path (A) edge  node{$e_0$} (B)
(C) edge  node{$e_1$} (B)
(B) edge  node{$e_2$} (D)
(A) edge node{$e_{-1}$} (F);

\end{tikzpicture}
    \end{minipage}%
    \end{figure}

\end{exem}

Given a graph $E = (E^0, E^1,r,s)$ we know that $E^0 = \left( \bigcup\limits_{v_i \in \Delta}^{.} Z_{v_i} \right) \bigcup\limits^{.} R$. Let us consider the connected subgraphs $E^{v_i} = (Z_{v_i}, s^{-1}(Z_{v_i}), r_{|s^{-1}(Z_{v_i})}, s_{|s^{-1}(Z_{v_i})} )$. 

\begin{cor}\label{last}
Let $E = (E^0, E^1,r,s)$ be a $P$-simple graph. Suppose that for each $i \in I$ there exists $n_i \in \mathbb{N}$ such that the graph $\mathcal{E}_{n_i}^{v_i}$ has finitely many vertices. Then, every representation $\varphi: C^*(E) \to B(H)$ of $C^*(E)$ is unitarily equivalent to a representation arising from a $E$-branching system.
\end{cor}

\begin{proof}
It follows by applying the previous result to each subgraph $E^{v_i}$.
\end{proof}

\begin{exem}
By the previous corollary, each representation of the graph below is unitary equivalent to a representation induced by a branching system.

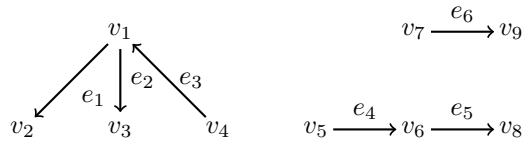
\centering
        \begin{tikzpicture}[->,auto,node distance=1.3 cm,
                    thick]
\tikzset{every state/.style={minimum size=0pt}} 

\node[state,inner sep=0.5pt,draw=none] (A){$v_1$};
\node[state,inner sep=0.5pt,draw=none] (B)[below of =A] {$v_3 $};
\node[state,inner sep=0.5pt,draw=none] (C)[left of =B] {$v_2$};
\node[state,inner sep=0.5pt,draw=none] (D)[right of =B] {$v_{4}$};
\node[state,inner sep=0.5pt,draw=none] (E)[right of =D] {$v_5$};
\node[state,inner sep=0.5pt,draw=none] (F)[right of =E]{$v_6$};
\node[state,inner sep=0.5pt,draw=none] (G)[ above of =F]{$v_7$};
\node[state,inner sep=0.5pt,draw=none] (H)[right of =F] {$v_8$};
\node[state,inner sep=0.5pt,draw=none] (I)[right of =G]{$v_9$};

\path 
(A) edge  node{$e_1$} (C)
(A) edge  node{$e_2$} (B)
(D) edge  node[right]{$e_3$} (A)
(E) edge  node{$e_4$}  (F)
(F) edge node{$e_5$} (H)
(G) edge node{$e_6$} (I);

\end{tikzpicture}
    
\end{exem}

The converse of the previous corollaries are note true. There are graphs such that every representation of their C*-algebras are unitarily equivalent to representations induced by a branching system but the graphs do not satisfy the hypothesis of the previous corolaries. We will show two examples.
		
\begin{exem}[2.2, \cite{B}]
		In \cite{B}, has been shown that every representation of the graph
		
\begin{center}
\begin{tikzpicture}[->,auto,node distance=2.5cm,
                    thick]
\tikzset{every state/.style={minimum size=0pt}} 
\tikzset{every loop/.style={min distance=10mm,in=0,out=80,looseness=20}} 

\node[state,inner sep=0.5pt,draw=none] (A)   {$v_{-1}$};
\node[state,inner sep=0.5pt,draw=none] (B) [right of =A]   {$v_0$};
\node[state,inner sep=0.5pt,draw=none] (C) [right of =B]   {$v_1$};
\node[state,inner sep=0.5pt,draw=none] (D) [right of =C]   {$...$};
\node[state,inner sep=0.5pt,draw=none] (F) [left of = A]   {$...$};

\path (A) edge  node{$e_0$} (B)
(B) edge  node{$e_1$} (C)
(C) edge  node{$e_2$} (D)
(F) edge node{$e_{-1}$} (A);

\end{tikzpicture}
\end{center}
is permutative. Then, by Theorem \ref{bpb} every representation of this graph is unitarily equivalent to a representation arising from a branching system. However, notice that there are no extreme edges and extreme vertices, so that no $\mathcal{E}_n$ does exist.\end{exem}

\begin{exem}\label{kk}
Let $E$ be the following graph
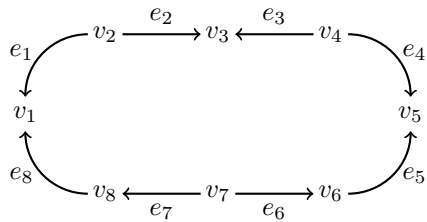
\begin{center}
\begin{tikzpicture}[->,auto,node distance=1.5cm,thick]
\tikzset{every state/.style={minimum size=0pt}} 
\tikzset{every loop/.style={min distance=10mm,looseness=12}} 

\node[state,inner sep=0.5pt,draw=none] (A)   {$v_5$};
\node[state,inner sep=0.5pt,draw=none] (B) [above left of =A]   {$v_{4}$};
\node[state,inner sep=0.5pt,draw=none] (C) [below left of =A]   {$v_{6}$};
\node[state,inner sep=0.5pt,draw=none] (D) [left of =B]   {$v_3$};
\node[state,inner sep=0.5pt,draw=none] (E) [left of =C]   {$v_7$};
\node[state,inner sep=0.5pt,draw=none] (F) [left of =D]   {$v_{2}$};
\node[state,inner sep=0.5pt,draw=none] (G) [left of =E]   {$v_{8}$};
\node[state,inner sep=0.5pt,draw=none] (H) [above left of =G]   {$v_1$};
 
\path (B) edge[in=90, out=0] node[right]{$e_4$}(A)
      (C) edge[in=270,out=0] node[right]{$e_5$} (A)
      (F) edge[in=90, out=180] node[left]{$e_1$}(H)
      (G) edge[in=270, out=180] node[left]{$e_8$}(H)
      (F) edge node{$e_2$}(D)
      (B) edge node[above]{$e_3$}(D)
      (E) edge node[below]{$e_6$}(C)
      (E) edge node{$e_7$}(G);
\end{tikzpicture}
\end{center}
and let $\varphi: C^*(E) \to B(H)$ an arbitrary representation. First choose arbitrary orthonormal total sets for the vertices $v$  such that $\#r^{-1}(v)\geq 2$, that is, the vertices $v_1,v_3$ and $v_5$. For this vertices choose  $B_{v_1} \subseteq H_{v_1}$, $B_{v_3} \subseteq H_{v_3}$, $B_{v_5} \subseteq H_{v_5}$. Now define $B_{e_i} := \varphi(S_{e_i})(B_{r(e_i)})$ for each $i = 1,2,3,4,5,8$. After this, choose $B_{v_2} := B_{e_1} \cup B_{e_2}$, $B_{v_4} := B_{e_3} \cup B_{e_4}$, $B_{v_8} := B_{e_8}$ and $B_{v_6} := B_{e_5}$; finally define $B_{e_7} := \pi(S_{e_7})(B_{v_8})$ and $B_{e_6} := \pi(S_{e_6})(B_{v_6})$. Finally define $B_{v_7} = B_{e_7} \cup B_{e_6}$. With this choices is clear that $\varphi$ is permutative. Of course this graphs doesn't satisfies the hypothesis of corollaries because $E$ isn't $P$-simple.
\end{exem}

\begin{obs}
The previous example may be generalized in a natural way, with analogous arguments, for every  graph $E$ which is a cycle with the property that $\{v \in E^0 \,\ | \,\ \#r^{-1}(v) = 2 \}\neq \emptyset$ for some vertex $v$. The hypothesis $\{v \in E^0 \,\ | \,\ \#r^{-1}(v) = 2 \}\neq\emptyset$ is important. For example, in Example \ref{exemploloop} the graph is a cycle, with $\{v \in E^0 \,\ | \,\ \#r^{-1}(v) = 2 \}=\emptyset$ for each vertex, and in this example there is shown a non permutative representation.
\end{obs}

\bibliographystyle{plain}
\bibliography{mybibliography}

\vspace{1.5pc}

Ben-Hur Eidt, Departamento de Matem\'{a}tica, Universidade Federal de Santa Catarina, Florian\'{o}polis, 88040-900, Brasil

Email: benhur96dt@gmail.com

\vspace{0.5pc}
Danilo Royer, Departamento de Matem\'{a}tica, Universidade Federal de Santa Catarina, Florian\'{o}polis, 88040-900, Brasil

Email: daniloroyer@gmail.com

\end{document}